\documentclass{article}%
\usepackage{graphicx}
\usepackage{amsmath}
\usepackage{amsfonts}
\usepackage{amssymb}%
\providecommand{\U}[1]{\protect\rule{.1in}{.1in}}
\newtheorem{theorem}{Theorem}
{}

\newtheorem{lemma}{Lemma}
{}
\newtheorem{notation}{Notation}

\newtheorem{summary}{Summary}
\newenvironment{proof}[1][Proof]{\textbf{#1.} }{\ \rule{0.5em}{0.5em}}

\begin{document}

\title{On the self-adjoint differential operator with the periodic matrix coefficients}
\author{O. A. Veliev\\{\small Dogus University, Esenkent 34775,\ Istanbul, Turkey.}\\\ {\small e-mail: oveliev@dogus.edu.tr}}
\date{}
\maketitle

\begin{abstract}
In this paper we consider the spectrum of the self-adjoint differential
operator $L$ generated \ by the differential expression of order $n$ with the
$m\times m$ periodic matrix coefficients, where $n$ and $m$ are respectively
odd and even integers and $n>1.$ We prove that the number of gaps in the
spectrum of $L$ is finite and find explicit estimation in term of coefficients
for the number of the gaps. Moreover, we find a condition on the norms of the
coefficients for which the spectrum is $(-\infty,\infty)$. Besides we
investigate the bands of the spectrum and prove that most of the real axis is
overlapped by $m$ bands.



\end{abstract}

\section{Introduction and Preliminary Facts}

Let $L$ be the differential operator generated in the space $L_{2}^{m}%
(-\infty,\infty)$ of the vector-valued functions by formally self-adjoint
differential expression
\begin{equation}
(-i)^{n}y^{(n)}(x)+%
{\textstyle\sum\limits_{v=2}^{n}}
P_{v}(x)y^{(n-v)}(x), \tag{1}%
\end{equation}
where $n$ is an odd integer and $n>1,$ $P_{v}\left(  x\right)  $ for
$v=2,3,...n$ are the $m\times m$ matrices with the $1$-periodic locally square
summable entries and $m$ is an even integer. It is well-known that (see [1, 2,
8, 10]) the spectrum $\sigma(L)$ of the operator $L$ is the union of the
spectra of the operators $L_{t}$ for $t\in(-1,1]$ generated in $L_{2}%
^{m}\left[  0,1\right]  $ by the expression (1) and the quasiperiodic
conditions
\begin{equation}
U_{\mathbb{\nu}}(y):=y^{(\mathbb{\nu})}\left(  1\right)  -e^{i\pi
t}y^{(\mathbb{\nu})}\left(  0\right)  =0,\text{ }\mathbb{\nu}=0,1,...,(n-1).
\tag{2}%
\end{equation}
Note that $L_{2}^{m}\left[  0,1\right]  $ is the set of vector-valued
functions $f=\left(  f_{1},f_{2},...,f_{m}\right)  $ with $f_{k}\in
L_{2}\left[  0,1\right]  $ for $k=1,2,...,m.$ The norm $\left\Vert
\cdot\right\Vert $ and inner product $(\cdot,\cdot)$ in $L_{2}^{m}\left[
0,1\right]  $ are defined by%
\[
\left\Vert f\right\Vert ^{2}=\int_{0}^{1}\left\vert f\left(  x\right)
\right\vert ^{2}dx,\text{ }(f,g)=\int_{0}^{1}\left\langle f\left(  x\right)
,g\left(  x\right)  \right\rangle dx,
\]
where $\left\vert \cdot\right\vert $ and $\left\langle \cdot,\cdot
\right\rangle $ are the norm and inner product in $\mathbb{C}^{m}.$ The
operators $L$ and $L_{t}$ are denoted by $L(0)$ and $L_{t}(0)$ if $P_{v}$ for
$v=2,3,...n$ are the zero matrices.

Let us discuss some preliminary results and describe the results of this
paper. Clearly, $\left(  2\pi k+\pi t\right)  ^{n}$ for $k\in\mathbb{Z}$ is
the eigenvalue of $L_{t}(0)$ with multiplicity $m$ and the corresponding
eigenspace is%
\begin{equation}
E_{k,t}=\left\{  ue^{i\pi\left(  2k+t\right)  x}:u\in\mathbb{C}^{m}\right\}  .
\tag{3}%
\end{equation}
The set $\left\{  \left(  2\pi k+\pi t\right)  ^{n}:t\in(-1,1],\text{ }%
k\in\mathbb{Z}\right\}  $ of the Bloch (generalized) eigenvalues of $L(0)$
overlap $m$ times (counting the multiplicity) the real axis.

First of all note that the case when $n$ is an even number $(n=2v)$
essentially differ from the case when $n$ is an odd number $(n=2v+1)$ in the
scalar $\left(  m=1\right)  $ and vectorial cases $\left(  m>1\right)  $. For
example, in the case $n=2$ and $m=1$ the finite zone potentials of the
Schrodinger operator are infinitely differentiable functions and have a
special form expressed by Riemann $\theta$ function (see [5, Chap. 8 and 9],
[6, Chap. 4] [3, 7]), while in [8] it was proved that if $n=2v+1$ and $m=1,$
then the spectrum of $L$ is $(-\infty,\infty)$. Moreover, this proof is very
simple and passes through for the case when both $n$ and $m$ are the odd
numbers (see [12]), due to that $nm$ is an odd number. However, this proof
does not walk for the case $n=2v+1$ and $m=2s,$ since $nm$ is an even number.
In [12], for this case, we proved that if the matrix $C=\int_{0}^{1}%
P_{2}\left(  x\right)  dx$ has at least one simple eigenvalue, then the number
of the gaps in $\sigma(L)$ is finite. Here we prove this statement without any
condition on $C$ and on the coefficients of (1). Besides, in this paper we
obtain the following additional results for the case $n=2v+1$ and $m=2s.$

1. We give an explicit estimation for the number of gaps in term of
\begin{equation}
M=%
{\textstyle\sum\limits_{v=2}^{n}}
\sup_{u}\left\Vert P_{v}u\right\Vert \leq%
{\textstyle\sum\limits_{v=2}^{n}}
\left(
{\textstyle\sum\limits_{i,j=1}^{m}}
\left\Vert p_{v,i,j}\right\Vert \right)  , \tag{4}%
\end{equation}
where $\sup$ is taken for $u\in\mathbb{C}^{m}$ and $\left\vert u\right\vert
=1$ and $p_{v,i,j}$ is the entry of $P_{v}.$ Namely, we prove that the number
of the gaps in $\sigma(L)$ is not greater than $m(2N-1)-1,$ where $N$ is the
smallest integer satisfying $N\geq\pi^{-2}M+1$.

2. We prove that if $M\leq\ \pi^{2}2^{-n+1/2}$, then $\sigma(L)=(-\infty
,\infty).$

3. We give a definition of the bands and band functions and prove that the
defined band functions continuously depend on $t\in(-1,1].$ Note that in the
case $n=2v$, $L_{t}$ is below-bounded self-adjoint operator. Therefore the
eigenvalues of $L_{t}$ can be numerated in the nondecreasing order
$\lambda_{1}(t)\leq\lambda_{2}(t)\leq\cdot\cdot\cdot$ and the function
$\lambda_{n}:(-1,1]\rightarrow\mathbb{R}$ is said to be the $n$-th band
function of $L$. This numeration and definition of the band function can not
be used for the odd $n,$ since then $L_{t}$ is not semibounded operator.

4. We prove that most of the real axis is overlapped by $m$ bands.

Finally note that, the method used in this paper is absolutely different from
the methods used in [11] and [12] for the cases $n=2$ and $n\geq2$ respectively.

\section{On the localization of the Bloch eigenvalues}

In this section we consider the localization of the eigenvalues of $L_{t}$.
Namely, we prove that if $\left\vert k\right\vert \geq\pi^{-2}M+1,$ where $M$
is defined in (4), then the $\delta_{k}(t)$ neighborhood
\begin{equation}
B(k,t)=\left(  (2\pi k+\pi t)^{n}-\delta_{k}(t),(2\pi k+\pi t)^{n}+\delta
_{k}(t)\right)  \tag{5}%
\end{equation}
of the eigenvalue $(2\pi k+\pi t)^{n}$ of $L_{t}(0)$ contains only $m$
eigenvalues of $L_{t},$ where $\delta_{k}(t)=\frac{3}{2}\pi^{n-2}M\left\vert
(2k+t)\right\vert ^{n-2}$.

First let us describe briefly the scheme of the proof of this statement. We
use the following notations and formulas. Let $\lambda(k,t,\varepsilon)$ be
the eigenvalue of the operator $L_{t,\varepsilon}=L_{t}(0)+\varepsilon
(L_{t}-L_{t}(0))$ satisfying
\begin{equation}
\lambda(k,t,\varepsilon)\in I(k,t) \tag{6}%
\end{equation}
and $\Psi_{\lambda(k,t,\varepsilon)}$ be a normalized eigenfunction of
$L_{t,}$ corresponding to the eigenvalue $\lambda(k,t,\varepsilon),$ where
$\varepsilon\in\lbrack0,1]$ and $I(k,t)=[(2\pi k+\pi t-\pi)^{n},(2\pi k+\pi
t+\pi)^{n}).$ Sometimes, for brevity, instead of $\Psi_{\lambda
(k,t,\varepsilon)}$ and $\lambda(k,t,\varepsilon)$ we write $\Psi_{\lambda}$
and $\lambda.$ The eigenfunction $\Psi_{\lambda(k,t,\varepsilon)}$ has a
decomposition
\begin{equation}
\Psi_{\lambda(k,t,\varepsilon)}=\sum\limits_{p\in\mathbb{Z}}\left(
\Psi_{\lambda(k,t,\varepsilon)},\varphi_{p,t}\right)  \varphi_{p,t}, \tag{7}%
\end{equation}
where $\varphi_{p,t}\in E_{p,t},$ $\left\Vert \varphi_{p,t}\right\Vert =1,$
$\varphi_{p,t}=u_{p}e^{i\left(  2\pi p+\pi t\right)  x},$ $E_{p,t}$ is defined
in (3) $u_{p}\in\mathbb{C}^{m}$ and $\left\vert u_{p}\right\vert =1.$
Multiplying the equation $L_{t,\varepsilon}\Psi_{\lambda}=\lambda\Psi
_{\lambda}$ by $\varphi_{k,t}$ and taking into account that (1) is a
self-adjoint differential expression we get
\begin{equation}
\left(  \lambda-\left(  2\pi k+\pi t\right)  ^{n}\right)  \left(
\Psi_{\lambda},\varphi_{k,t}\right)  =\varepsilon\sum\limits_{\nu=2}%
^{n}\left(  i2\pi k+i\pi t\right)  ^{n-v}(\Psi_{\lambda}^{{}},P_{\nu}%
\varphi_{k,t}), \tag{8}%
\end{equation}
where $\lambda=\lambda(k,t,\varepsilon).$ We use this formula as follows.
First note that we have the following estimations for the right hand side of
(8). If $k\in\mathbb{Z}\backslash\left\{  0\right\}  ,$ then
\begin{equation}
\left\vert \varepsilon\sum\limits_{\nu=2}^{n}\left(  i2\pi k+i\pi t\right)
^{n-v}(\Psi_{\lambda}^{{}},P_{\nu}\varphi_{k,t})\right\vert \leq\left(
\left\vert 2k+t\right\vert \pi\right)  ^{n-2}M \tag{9}%
\end{equation}
for all $\varepsilon\in\lbrack0,1]$. In case $k=0$ we have
\begin{equation}
\left\vert \varepsilon\sum\limits_{\nu=2\pi}^{n}\left(  i\pi t\right)
^{n-v}(\Psi_{\lambda}^{{}},P_{\nu}\varphi_{0,t})\right\vert \leq\pi^{n-2}M
\tag{10}%
\end{equation}
for all $\varepsilon\in\lbrack0,1]$ and $t\in(-1,1].$ Moreover, we prove that
\begin{equation}
\left\vert \left(  \Psi_{\lambda(k,t,\varepsilon)},\varphi_{k,t}\right)
\right\vert >\frac{2}{3}. \tag{11}%
\end{equation}
Then using (9)-(11) in (8) we obtain that the eigenvalue $\lambda
(k,t,\varepsilon)$ satisfying (6) is contained in $B(k,t)$ if $\left\vert
k\right\vert \geq\pi^{-2}M+1$ (see Theorem 1). To obtain (11) we use the
obvious equality
\begin{equation}
1=\sum\limits_{p\in\mathbb{Z}}\left\vert \left(  \Psi_{\lambda(k,t,\varepsilon
)},\varphi_{p,t}\right)  \right\vert ^{2} \tag{12}%
\end{equation}
(see (7)) and prove the inequalities
\begin{equation}
\sum_{p>k}\left\vert \left(  \Psi_{\lambda(k,t,\varepsilon)},\varphi
_{p,t}\right)  \right\vert ^{2}<\frac{5}{64}, \tag{13}%
\end{equation}%
\begin{equation}
\sum_{p<-k}\left\vert \left(  \Psi_{\lambda(k,t,\varepsilon)},\varphi
_{p,t}\right)  \right\vert ^{2}<\frac{1}{48} \tag{14}%
\end{equation}
and%
\begin{equation}
\sum_{-k\leq p<k}\left\vert \left(  \Psi_{\lambda(k,t,\varepsilon)}%
,\varphi_{p,t}\right)  \right\vert ^{2}<\frac{27}{64}. \tag{15}%
\end{equation}
The estimations (13), (14) and (15) are proved in Lemma 1 and Lemma 2 respectively.

The case $M\leq\ \pi^{2}2^{-n+1/2}$ is considered in Theorem 2 in the similar
way. In the estimations we use the following inequalities. If $ab>0$ and
$n\geq3,$ then
\begin{equation}
\left\vert \frac{a^{n-2}}{a^{n}-b^{n}}\right\vert <\left\vert \frac
{1}{(a-b)(a+b)}\right\vert , \tag{16}%
\end{equation}
where $n$ is an odd number. If $l\geq2$ and $s\geq p,$ then
\begin{equation}
\sum_{j=p}^{\infty}\frac{1}{(2j+1)^{l}}<\sum_{j=p}^{s-1}\frac{1}{(2p+1)^{l}%
}+\frac{1}{2^{l}}\int_{s}^{\infty}\frac{1}{x^{l}}dx. \tag{17}%
\end{equation}

\begin{lemma}
\bigskip If $\left\vert k\right\vert \geq\pi^{-2}M,$ then (13) and (14) hold.
\end{lemma}

\begin{proof}
We consider the case $k>0.$ The case $k<0$ can be considered in the same way.
It follows from (6) that
\[
\left\vert \lambda(k,t,\varepsilon)-(2\pi p+\pi t)^{n}\right\vert >\left\vert
(2\pi k+\pi t+\pi)^{n}-(2\pi p+\pi t)^{n}\right\vert
\]
for $p>k$ and
\[
\left\vert \lambda(k,t,\varepsilon)-(2\pi p+\pi t)^{n}\right\vert
\geq\left\vert (2\pi k+\pi t-\pi)^{n}-(2\pi p+\pi t)^{n}\right\vert
\]
for $p<k$. Using these inequalities and the relations obtained from (8) and
(9) by replacing $k$ with $p\neq0$ we obtain
\begin{equation}
\left\vert \left(  \Psi_{\lambda(k,t,\varepsilon)},\varphi_{p,t}\right)
\right\vert ^{2}<\frac{\pi^{-4}M^{2}\left(  \left\vert 2p+t\right\vert
^{n-2}\right)  ^{2}}{\left(  (2k+t+1)^{n}-(2p+t)^{n}\right)  ^{2}} \tag{18}%
\end{equation}
for $p>k$ and
\begin{equation}
\left\vert \left(  \Psi_{\lambda(k,t,\varepsilon)},\varphi_{p,t}\right)
\right\vert ^{2}\leq\frac{\pi^{-4}M^{2}\left(  \left\vert 2p+t\right\vert
^{n-2}\right)  ^{2}}{\left(  (2k+t-1)^{n}-(2p+t)^{n}\right)  ^{2}} \tag{19}%
\end{equation}
for $p<k.$ In case $p=0,$ instead of (9) using (10) we get the formulas
\begin{equation}
\left\vert \left(  \Psi_{\lambda(k,t,\varepsilon)},\varphi_{0,t}\right)
\right\vert ^{2}<\frac{\pi^{-4}M^{2}}{\left(  (2k+t+1)^{n}-t^{n}\right)  ^{2}}
\tag{20}%
\end{equation}
for $k<0$ and
\begin{equation}
\left\vert \left(  \Psi_{\lambda(k,t,\varepsilon)},\varphi_{0,t}\right)
\right\vert ^{2}\leq\frac{\pi^{-4}M^{2}}{\left(  (2k+t-1)^{n}-t^{n}\right)
^{2}} \tag{21}%
\end{equation}
for $k>0$ instead of (18) and (19).

First we prove (13). Using (16), where $a=(2p+t)^{n}$ and $b=(2k+1+t),$ and
the relation $t\in(-1,1]$ in (18) we obtain
\begin{equation}
\sum_{p>k}\left\vert \left(  \Psi_{\lambda},\varphi_{p,t}\right)  \right\vert
^{2}<\sum_{p>k}\frac{\pi^{-4}M^{2}}{(2p-2k-1)^{2}(2p+2k-1)^{2}} \tag{22}%
\end{equation}
for all $t\in(-1,1].$ Now we estimate the right side of (22) as follows. Since
$p>k\geq1$ we have $(2p+2k-1)^{2}>16k^{2}$. On the other hand, using (17) for
$l=2,$ $p=0$ and $s=5$ by direct calculations one can easily verify that
\begin{equation}
\sum_{p>k}\frac{1}{(2p-2k-1)^{2}}<\sum_{j=0}^{4}\frac{1}{(2j+1)^{2}}+\frac
{1}{4}\int_{5}^{\infty}\frac{1}{x^{2}}dx<\frac{5}{4}. \tag{23}%
\end{equation}
Using these inequalities in (22) we obtain
\begin{equation}
\sum_{p>k}\left\vert \left(  \Psi_{\lambda(k,t,\varepsilon)},\varphi
_{p,t}\right)  \right\vert ^{2}<\frac{5\pi^{-4}M^{2}}{64k^{2}}. \tag{24}%
\end{equation}
Now (13) follows from (24), since $\left\vert k\right\vert \geq\pi^{-2}M$.

Now we prove (14). If $p<-k,$ where $k\geq1,$ then both $(2k-1+t)^{n}$ and
$-(2p+t)^{n}$ are the positive numbers. Therefore it follows from (19) that
\begin{equation}
\sum_{p<-k}\left\vert \left(  \Psi_{\lambda(k,t,\varepsilon)},\varphi
_{p,t}\right)  \right\vert ^{2}\leq\sum_{p<-k}\frac{\pi^{-4}M^{2}}{(2p+t)^{4}%
}\leq\sum_{p<-k}\frac{\pi^{-4}M^{2}}{(2p+1)^{4}} \tag{25}%
\end{equation}
for all $t\in(-1,1].$ Now using the obvious inequality
\begin{equation}
\sum_{j=k}^{\infty}\frac{1}{(2j+1)^{4}}\leq\frac{1}{16}\int_{k}^{\infty}%
\frac{1}{x^{4}}dx=\frac{1}{48\left\vert k\right\vert ^{3}} \tag{26}%
\end{equation}
(see (17) for $l=4$ and $p=s=k$ ) we get
\begin{equation}
\sum_{p<-k}\left\vert \left(  \Psi_{\lambda(k,t,\varepsilon)},\varphi
_{p,t}\right)  \right\vert ^{2}\leq\frac{\pi^{-4}M^{2}}{48\left\vert
k\right\vert ^{3}}. \tag{27}%
\end{equation}
Thus (14) follows from (27), since $k\geq\pi^{-2}M$ and $k\geq1.$
\end{proof}

\begin{lemma}
\bigskip If$\ \left\vert k\right\vert \geq\pi^{-2}M+1,$ then (15) holds.
\end{lemma}

\begin{proof}
We consider (15) for the cases $p\in($\ $0,k),$ $p=0$, $p\in(-k,0)$ and $p=-k$
separately, where $k\geq2.$ The case $k\leq-2$ can be considered in the same
way. If $p\in(0,k)$ then $(2p+t)\geq0.$ Hence using (16) for $a=2k-1+t$ and
$b=2p+t$ and the relations $t\in(-1,1]$, $p\geq1$ in (19) we obtain
\begin{equation}
\sum_{p\in(0,k)}\left\vert \left(  \Psi_{\lambda(k,t,\varepsilon)}%
,\varphi_{p,t}\right)  \right\vert ^{2}\leq\sum_{p\in(0,k)}\frac{\pi^{-4}%
M^{2}}{(2k-2p-1)^{2}(2k-1)^{2}}. \tag{28}%
\end{equation}
Now using (23) and the inequality $2k-1>2\pi^{-2}M$ we obtain
\begin{equation}
\sum_{p\in(0,k)}\left\vert \left(  \Psi_{\lambda(k,t,\varepsilon)}%
,\varphi_{p,t}\right)  \right\vert ^{2}<\frac{5}{4}\frac{\pi^{-4}M^{2}%
}{(2k-1)^{2}}<\frac{5}{16}. \tag{29}%
\end{equation}
\ 

Now consider the case $p=0.$ One can easily verify that the function $f$
defined by $f(t)=(2k-1+t)^{n}-t^{n}$ gets its minimum value on $[-1,1]$ at
$t=-1.$ Therefore we have $(2k-1+t)^{n}-t^{n}>(2k-2)^{n}$ for all
$t\in(-1,1].$ Thus using this inequality and the inequalities
\begin{equation}
2k-2\geq2\pi^{-2}M,\text{ }2k-2\geq2 \tag{30}%
\end{equation}
and $n\geq3$ in (21) we obtain
\begin{equation}
\left\vert \left(  \Psi_{\lambda(k,t,\varepsilon)},\varphi_{0,t}\right)
\right\vert ^{2}\leq\frac{\pi^{-4}M^{2}}{(2k-2)^{2n}}\leq\frac{\pi^{-4}M^{2}%
}{16(2k-2)^{2}}\leq\frac{1}{64}. \tag{31}%
\end{equation}

If $p\in(-k,0),$ then both $(2k-1+t)^{n}$ and $-(2p+t)^{n}$ are the positive
numbers. Therefore in the both cases $(2k-1+t)^{n}$ $\geq$ $-(2p+t)^{n}$ and
$(2k-1+t)^{n}\leq-(2p+t)^{n}$ we have
\[
\frac{\left\vert 2p+t\right\vert ^{n-2}}{(2k-1+t)^{n}-(2p+t)^{n}}<\frac
{1}{(2k-1+t)^{2}}\leq\frac{1}{(2k-2)^{2}}%
\]
for all $t\in(-1,1].$ Using this, (19) and (30) we obtain
\begin{equation}
\sum_{p\in(-k,0)}\left\vert \left(  \Psi_{\lambda(k,t,\varepsilon)}%
,\varphi_{p,t}\right)  \right\vert ^{2}<\frac{\pi^{-4}\left(  k-1\right)
M^{2}}{(2k-2)^{4}}\leq\frac{1}{16}, \tag{32}%
\end{equation}
since the number of integer in $(-k,0)$ is $k-1.$

It remains to consider the case $p=-k.$ In this case using the obvious
inequality $(2k-1+t)^{n}-(-2k+t)^{n}\geq-(-2k+t)^{n}$ and (30) in (19) we
obtain
\[
\left\vert \left(  \Psi_{\lambda(k,t,\varepsilon)},\varphi_{-k,t}\right)
\right\vert ^{2}\leq\frac{\pi^{-4}M^{2}}{(2k-1)^{4}}\leq\frac{1}{36}%
\]
for all $t\in(-1,1].$ Now using this inequality, (29), (31) and (32) we obtain
(15) by direct calculations.
\end{proof}

Now we are ready to prove the following theorem about the localization of the
Bloch eigenvalues.

\begin{theorem}
Let $N$ be the smallest integer satisfying $N\geq\pi^{-2}M+1$ and
\[
A(N,t)=[(-2\pi N+\pi+\pi t)^{n},(2\pi N-\pi+\pi t)^{n}).
\]

$(a)$ The spectrum of $L_{t}$ is contained in the union of the intervals
$A(N,t)$ and $B(k,t)$ for $\left\vert k\right\vert \geq N$ , where $B(k,t)$ is
defined in (5).

$(b)$ The intervals $A(N,t)$ and $B(k,t)$ for $\left\vert k\right\vert \geq N$
are pairwise disjoint.

$(c)$ Each of the intervals $B(k,t)$ for $\left\vert k\right\vert \geq N$
contains only $m$ eigenvalues of $L_{t}$. The interval $A(N,t)$ contains only
$(2N-1)m$ eigenvalues of $L_{t}$.
\end{theorem}

\begin{proof}
$(a)$ It follows from Lemma 1 and Lemma 2 that the eigenvalue $\lambda
(k,t,\varepsilon)$ of $L_{t}$ lying in the interval $I(k,t)$ is contained in
the interval $B(k,t),$ where $I(k,t)$ is defined in (6) (see the beginning of
this section). Since the intervals $A(N,t)$ and $I(k,t)$ for $\left\vert
k\right\vert \geq N$ is a cover of $(-\infty,\infty),$ all eigenvalues of
$L_{t}$ is contained in the union of the intervals $A(N,t)$ and $B(k,t)$ for
$\left\vert k\right\vert \geq N.$

$(b)$ First let us prove that the intervals $B(k,t)$ and $B(k+1,t)$ for $k\geq
N$ has no common point. In the other words we prove that the distance between
the centers of $B(k,t)$ and $B(k+1,t)$ is greater than $\delta_{k}%
(t)+\delta_{k+1}(t),$ where $B(k,t)$ and $\delta_{k}(t)$ are defined in (5).
It can be proved by using (16) as follows. The inequality (16) can be written
in the form $a^{n}-b^{n}>(a-b)a^{n-2}(a+b).$ Now instead of $a$ and $b$ taking
$(2\pi k+2\pi+\pi t)$ and $b=(2\pi k+\pi t)$ one can easily verify by direct
calculations that $(2k\pi+2\pi+\pi t)^{n}-(2\pi k+\pi t)^{n}>\delta
_{k}(t)+\delta_{k+1}(t)$ for $k\geq N$ and $t\in(-1,1].$ Thus $B(k,t)$ and
$B(k+1,t)$ for $\left\vert k\right\vert \geq N$ has no common point. In the
same way, we prove that $B(N,t)$ for $\left\vert k\right\vert \geq N$ has no
common point with the interval $A(N,t).$ The same holds for $k\leq-N.$

$(c)$ It follows from $(b)$ that the gaps
\begin{equation}
S(k,t):=\left[  (2\pi k+\pi t)^{n}+\delta_{k}(t),(2\pi k+2\pi+\pi
t)^{n}-\delta_{k+1}(t)\right]  \tag{33}%
\end{equation}
between $B(k,t)$ and $B(k+1,t)$ for $k\geq N$ belong to the resolvent set of
the operators $L_{t,\varepsilon}$ for all $\varepsilon\in\lbrack0,1].$ Let
$\Gamma(k,t)$ be the circle passing through the middle points of the gaps
$S(k,t)$ and $S(k+1,t).$ Then $\Gamma(k,t)$ enclose the interval $B(k+1,t)$
and lies in the resolvent set of the operators $L_{t,\varepsilon}$ for all
$\varepsilon\in\lbrack0,1].$ Therefore, taking into account that the family
$L_{t,\varepsilon}$ is halomorphic with respect to $\varepsilon,$ we obtain
that the number of eigenvalues (counting the multiplicity) of the operators
$L_{t,0}=L_{t}(0)$ and $L_{t,1}=L_{t}$ lying inside $\Gamma(k,t)$ are the same
(see [4, Chap. 7]). Since the operator $L_{t,0}$ has only $m$ eigenvalues
lying inside $\Gamma(k,t)$, the operator $L_{t,1}$ also has only $m$
eigenvalues lying inside $\Gamma(k,t).$ Thus $B(k,t)$ for $k>N$ contains only
$m$ eigenvalues of $L_{t}.$ Instead of $\Gamma(k,t)$ using the circle passing
through the middle points of $S(N,t)$ and of the gap
\begin{equation}
D(N,t)=[(2\pi N-\pi+\pi t)^{n},(2\pi N+\pi t)^{n}-\delta_{N}(t))\tag{34}%
\end{equation}
between $A(N,t)$ and $B(N,t)$ we obtain that $B(N,t)$ also contains only $m$
eigenvalues of $L_{t}.$ In the same way we prove this statement for $k\leq-N.$

Instead of $\Gamma(k,t)$ using the circle passing through the middle points of
$D(N,t)$ and of the gap between $A(N,t)$ and $B(-N,t)$ we obtain that $A(N,t)$
contains only $(2N-1)m$ eigenvalues of $L_{t}$.
\end{proof}

Now we consider the localization of the Bloch eigenvalues for the case
\begin{equation}
M\leq\ \pi^{2}2^{-n+1/2}. \tag{35}%
\end{equation}

\begin{theorem}
If (35) holds, then the spectrum of $L_{t}$ is contained in the intervals
\[
C(t)=(\left(  \pi t\right)  ^{n}-\frac{1}{5}\pi^{n},\left(  \pi t\right)
^{n}+\frac{1}{5}\pi^{n}),
\]
\[
C(1,t)=\left(  \left(  2\pi+\pi t\right)  ^{n}-\frac{3}{10}\left\vert
2+t\right\vert ^{n-2}\pi^{n},\left(  2\pi+\pi t\right)  ^{n}+\frac{3}%
{10}\left\vert 2+t\right\vert ^{n-2}\pi^{n}\right)  ,
\]%
\[
C(-1,t)=\left(  \left(  \pi t-2\pi\right)  ^{n}-\frac{3}{10}\left\vert
t-2\right\vert ^{n-2}\pi^{n},\left(  \pi t-2\pi\right)  ^{n}+\frac{3}%
{10}\left\vert t-2\right\vert ^{n-2}\pi^{n}\right)
\]
and $B(k,t)$ for $\left\vert k\right\vert >1.$ These intervals are pairwise
disjoint and each of these intervals contains only $m$ eigenvalues of $L_{t}.$
\end{theorem}

\begin{proof}
If (35) holds, then $\pi^{-2}M<1$ and $\pi^{-2}M+1<2.$ Hence the number $N$
defined in Theorem 1 is $2$ and Theorem 1 continue to hold for $N\geq2.$
Therefore we need to consider the cases $k=1,$ $k=0$ and $k=-1.$ In other
words we study the eigenvalues $\lambda(0,t,\varepsilon),$ $\lambda
(1,t,\varepsilon)$ and $\lambda(-1,t,\varepsilon)$ lying respectively in the
intervals $I(0,t),$ $I(1,t)$ and $I(-1,t).$ First we consider $\lambda
(1,t,\varepsilon).$ The eigenvalue $\lambda(-1,t,\varepsilon)$ can be
considered in the same way. By Lemma 1 the inequalities (13) and (14) hold for
all $\left\vert k\right\vert \geq1$. Thus we have
\begin{equation}
\sum_{\left\vert p\right\vert >1}\left\vert \left(  \Psi_{\lambda
(1,t,\varepsilon)},\varphi_{p,t}\right)  \right\vert ^{2}<\frac{5}{64}%
+\frac{1}{48}. \tag{36}%
\end{equation}
Hence, it remains to estimate $\left\vert \left(  \Psi_{\lambda
(1,t,\varepsilon)},\varphi_{0,t}\right)  \right\vert $ and $\left\vert \left(
\Psi_{\lambda(1,t,\varepsilon)},\varphi_{-1,t}\right)  \right\vert .$ It
follows from (21) and (19) that
\begin{equation}
\left\vert \left(  \Psi_{\lambda(1,t,\varepsilon)},\varphi_{0,t}\right)
\right\vert ^{2}\leq\frac{\pi^{-4}M^{2}}{\left(  (1+t)^{n}-t^{n}\right)  ^{2}}
\tag{37}%
\end{equation}
and
\begin{equation}
\left\vert \left(  \Psi_{\lambda(1,t,\varepsilon)},\varphi_{-1,t}\right)
\right\vert ^{2}\leq\frac{\left(  \left\vert -2+t\right\vert ^{n-2}\right)
^{2}M^{2}\pi^{-4}}{\left(  (1+t)^{n}-(-2+t)^{n}\right)  ^{2}}. \tag{38}%
\end{equation}
First let us consider the right side of (37). One can easily verify that for
odd $n\geq3$ the function $f(t)=(1+t)^{n}-t^{n}$ gets its minimum value on
$[-1,1]$ at $t=-1/2.$ It means that $(1+t)^{n}-t^{n}\geq2^{1-n}$ for all
$t\in(-1,1].$ Therefore using (37) and the inequality $M^{2}\pi^{-4}%
\leq2^{1-2n}$ (see (35)) we obtain
\[
\left\vert \left(  \Psi_{\lambda(1,t,\varepsilon)},\varphi_{0,t}\right)
\right\vert ^{2}\leq2^{2n-2}M^{2}\pi^{-4}\leq\frac{1}{2}.
\]
Now consider the right side of (38). Since both $(1+t)^{n}$ and $-(-2+t)^{n}$
are positive numbers, $\left(  (-2+t)^{n}-(1+t)^{n}\right)  ^{2}%
\geq(-2+t)^{2n}$ and hence
\[
\left\vert \left(  \Psi_{\lambda(1,t,\varepsilon)},\varphi_{-1,t}\right)
\right\vert ^{2}\leq\frac{M^{2}\pi^{-4}}{(-2+t)^{4}}\leq M^{2}\pi^{-4}%
\leq2^{1-2n}\leq\frac{1}{32}%
\]
for $n\geq2.$ These inequalities with (36) and (12) implies that
\[
\left\vert \left(  \Psi_{\lambda(1,t,\varepsilon)},\varphi_{1,t}\right)
\right\vert ^{2}>1-\frac{1}{2}-\frac{1}{32}-\frac{1}{48}-\frac{5}{64}%
=\frac{71}{192},\text{ }\left\vert \left(  \Psi_{\lambda},\varphi
_{1,t}\right)  \right\vert >\frac{3}{5}.
\]
Now using this and (9) for $k=1$ in (8) and (35) we obtain
\[
\left\vert \lambda(1,t,\varepsilon)-\left(  2\pi+\pi t\right)  ^{n}\right\vert
<\frac{5}{3}\left\vert 2+t\right\vert ^{n-2}\pi^{n-2}M\leq\frac{3}%
{10}\left\vert 2+t\right\vert ^{n-2}\pi^{n}.
\]
\ 

In remains to consider $\lambda(0,t,\varepsilon).$ For this instead of (13)
and (14) we prove the following inequalities
\begin{equation}
\sum_{p>0}\left\vert \left(  \Psi_{\lambda(0,t,\varepsilon)},\varphi
_{p,t}\right)  \right\vert ^{2}<\frac{49}{48}2^{1-2n} \tag{39}%
\end{equation}
and
\begin{equation}
\sum_{p<0}\left\vert \left(  \Psi_{\lambda(0,t,\varepsilon)},\varphi
_{p,t}\right)  \right\vert ^{2}<\frac{49}{48}2^{1-2n}. \tag{40}%
\end{equation}
First we prove (39). Using (18) for $k=0$ and (16) we obtain%
\[
\sum_{p>0}\left\vert \left(  \Psi_{\lambda(0,t,\varepsilon)},\varphi
_{p,t}\right)  \right\vert ^{2}<\sum_{p>0}\frac{\pi^{-4}M^{2}}{(2p-1)^{4}}.
\]
Therefore (35) and the obvious relation
\[
\sum_{s=0}^{\infty}\frac{1}{(2s+1)^{4}}\leq1+\frac{1}{16}\int\limits_{1}%
^{\infty}\frac{1}{x^{4}}dx=\frac{49}{48}%
\]
give the proof of (39). Instead of (18) using (19) and repeating the proof
(39) we get the proof of (40). \ Now from (39), (40) and (12) we obtain
\begin{equation}
\left\vert \left(  \Psi_{\lambda(0,t,\varepsilon)},\varphi_{0,t}\right)
\right\vert >\sqrt{1-\frac{49}{48}2^{2-2n}}\geq\sqrt{1-\frac{49}{48\times
16}^{{}}}>\frac{9}{10}. \tag{41}%
\end{equation}
Using (41) and (10) in (8) for $k=0$ we obtain
\[
\left\vert \lambda(0,t,\varepsilon)-\left(  \pi t\right)  ^{n}\right\vert
\leq\frac{10}{9}\pi^{n-2}M\leq\frac{10}{9}\pi^{n}2^{-n+1/2}<\frac{1}{5}\pi
^{n}.
\]
Thus the eigenvalues of $L_{t}$ is contained in the intervals $C(-1,t),$
$C(t),$ $C(1,t)$ and $B(k,t)$ for $\left\vert k\right\vert >1.$

By Theorem 1$(b)$ to prove that these intervals are pairwise disjoint we need
only prove that $C(t)\cap C(1,t)=\varnothing$ and $C(t)\cap
C(-1,t)=\varnothing.$ We prove the first of them. The proof of the second of
them is the same. For this we prove that the left end point of $C(1,t)$ is on
the right of the right end point of $C(t).$ Since $\left\vert 2+t\right\vert
^{n-2}\leq\left\vert 2+t\right\vert ^{n}$ it is enough to show that
$f(t):=\frac{7}{10}\left(  2+t\right)  ^{n}-t^{n}-\frac{1}{5}>0$ for all
$t\in(-1,1].$ This can be verified by calculating the values of$\ f(t)$ at the
end points of $[-1,1]$ and at the critical point $t=-\frac{14}{17}$. Repeating
the proof of Theorem 1$(c)$ we obtain that each of the intervals $C(-1,t),$
$C(t),$ $C(1,t)$ and $B(k,t)$ for $\left\vert k\right\vert >1.$ contains only
$m$ eigenvalues of $L_{t}.$
\end{proof}

\section{On the spectrum of $L$}

First of all using Theorem 1 we give a numerations of the eigenvalues of
$L_{t}$ for\ $t\in(-1,1]$ by the elements of $\mathbb{Z}\backslash\left\{
0\right\}  $ and define the band functions of $L.$

\begin{notation}
Let us denote the eigenvalues of $L_{t}$ lying on the right and left of
$D(N,t)$ by $\lambda_{1}(t)\leq\lambda_{2}(t)\leq\cdot\cdot\cdot.$ and
$\lambda_{-1}(t)\geq\lambda_{-2}(t)\geq\cdot\cdot\cdot$\ respectively, where
$D(N,t)$ is defined in (34). In this way we obtain the numeration of of the
eigenvalues of $L_{t}$ by the elements of $\mathbb{Z}\backslash\left\{
0\right\}  $ such that%
\[
\cdot\cdot\cdot\leq\lambda_{-2}(t)\leq\lambda_{-1}(t)\leq\lambda_{1}%
(t)\leq\lambda_{2}(t)\leq\cdot\cdot\cdot.
\]
The set $\Delta_{s}=\left\{  \lambda_{s}(t):t\in(-1,1]\right\}  $ and the
function $\lambda_{s}:(-1,1]\rightarrow\mathbb{R}$ are said to be the $s$-th
band and $s$-th band function of $L.$ In these notations the eigenvalues of
$L_{t}$ lying in $B(k,t),$ $A(N,t)$ and $B(-k,t)$ for $\left\vert k\right\vert
\geq N$ are respectively%
\[
\lambda_{(k-N)m+1}(t)\leq\lambda_{(k-N)m+2}(t)\leq\cdot\cdot\cdot\leq
\lambda_{(k-N)m+m}(t),
\]%
\[
\lambda_{-\left(  2N-1\right)  m}(t)\leq\lambda_{-\left(  2N-1\right)
m+1}(t)\cdot\cdot\cdot\leq\lambda_{-1}(t)
\]
and
\[
\lambda_{-\left(  N+k\right)  m}(t)\leq\lambda_{-\left(  N+k\right)
m+1}(t)\cdot\cdot\cdot\leq\lambda_{-\left(  N+k\right)  m+m-1}(t).
\]

\end{notation}

Now we prove that these band functions continuously depend on $t.$ For this we
use the following well-known statements (see for example [9] Chap. 3)
formulated here as summaries.

\begin{summary}
The eigenvalues of $L_{t}$ are the roots of the characteristic determinant
\[
\Delta(\lambda,t)=\det(Y_{j}^{(v-1)}(1,\lambda)-e^{i\pi t}Y_{j}^{(v-1)}%
(0,\lambda))_{j,v=1}^{n}=
\]%
\[
e^{inm\pi t}+f_{1}(\lambda)e^{i(nm-1)\pi t}+f_{2}(\lambda)e^{i(nm-2)\pi
t}+\cdot\cdot\cdot+f_{nm-1}(\lambda)e^{i\pi t}+1
\]
which is a polynomial of $e^{i\pi t}$\ with entire coefficients $f_{1}%
(\lambda),f_{2}(\lambda),...$ , where

$Y_{1}(x,\lambda),Y_{2}(x,\lambda),\ldots,Y_{n}(x,\lambda)$ are the solutions
of the matrix equation
\[
(-i)^{n}Y^{\left(  n\right)  }+P_{2}(x)Y^{(n-2)}+P_{3}(x)Y^{(n-3)}+\cdot
\cdot\cdot+P_{2s}(x)Y=\lambda Y
\]
satisfying $Y_{k}^{(j)}(0,\lambda)=O$ for $j\neq k-1$ and $Y_{k}%
^{(k-1)}(0,\lambda)=I$. Here $O$ and $I$ are $m\times m$ zero and identity
matrices respectively.
\end{summary}

\begin{summary}
The Green's function of $L_{t}-\lambda I$ is defined by formula%
\[
G(x,\xi,\lambda,t)=g(x,\xi,\lambda)-\frac{1}{\Delta(\lambda,t)}\sum
\limits_{j,v=1}^{n}Y_{j}(x,\lambda)V_{jv}(x,\lambda)U_{v}(g),
\]
where $g$ does not depend on $t$ and $V_{jv}$ is the transpose of that $m$th
order matrix consisting of the cofactor of the element $U_{v}(Y_{j})$ in
$\det(U_{v}(Y_{j}))_{j,v=1}^{n}.$ Hence the entries of the matrices
$V_{jv}(x,\lambda)$ and $U_{v}(g)$ either do not depend on $t$ or have the
form $u^{(v)}(1,\lambda)-e^{i\pi t}u^{(v)}(0,\lambda)$ and $h(1,\xi
,\lambda)-e^{i\pi t}h(0,\xi,\lambda)$ respectively, where the functions $u$
and $h$ do not depend on $t.$ Moreover, the Green's function is a meromorphic
matrix function of the parameter $\lambda$ and only eigenvalues of $L_{t}$ can
be poles of this function.
\end{summary}

Using these summaries we prove that for each $k\in\mathbb{Z}\backslash\left\{
0\right\}  $\ the function $\lambda_{k}$ defined in Notation 1 is continuous
at each point $t_{0}\in(-1,1].$ For this we introduce the following notations.

\begin{notation}
In Notation 1 the eigenvalues of the operator $L_{t_{0}}$ are denoted by
counting the multiplicity. Let us denote by $\mu_{1}(t_{0}),\mu_{2}%
(t_{0}),...$ the eigenvalues of $L_{t_{0}}$ lying on the right of $D(N,t)$
without counting the multiplicity, where $D(N,t)$ is defined in (34). In the
other words, $\mu_{1}(t_{0})<\mu_{2}(t_{0})<\cdot\cdot\cdot$ are the
eigenvalues of $L_{t_{0}}$ with the multiplicities $k_{1},k_{2},...$
respectively. Since $\lambda_{k}(t_{0})\rightarrow\infty$ as $k\rightarrow
\infty,$ for each $k$ there exists $p$ such that $k\leq k_{1}+k_{2}+\cdot
\cdot\cdot+k_{p}.$ This notation with the Notation 1 implies that
\[
\lambda_{1}(t_{0})=\lambda_{2}(t_{0})=\cdot\cdot\cdot=\lambda_{s_{1}}%
(t_{0})=\mu_{1}(t_{0}),
\]%
\[
\lambda_{s_{1}+1}(t_{0})=\lambda_{s_{1}+2}(t_{0})=\cdot\cdot\cdot
=\lambda_{s_{2}}(t_{0})=\mu_{2}(t_{0}),
\]
and
\[
\lambda_{s_{p-1}+1}(t_{0})=\lambda_{s_{p-1}+2}(t_{0})=\cdot\cdot\cdot
=\lambda_{s_{p}}(t_{0})=\mu_{p}(t_{0}),
\]
where $s_{p}=k_{1}+k_{2}+\cdot\cdot\cdot+k_{p}$ and $k\leq s_{p}.$
\end{notation}

Now we are ready to prove the following two theorems.

\begin{theorem}
$(a)$ For every $r>0$ satisfying the inequality
\[
r<\tfrac{1}{2}\min_{j=1,2,...p}\left(  \mu_{j+1}(t_{0})-\mu_{j}(t_{0})\right)
\]
there exists $\delta>0$ such that the operators $L_{t}$ for $t\in(t_{0}%
-\delta,t_{0}+\delta)$ have \ $k_{j}$ eigenvalues in the intervals $\left(
\mu_{j}(t_{0})-r,\mu_{j}(t_{0})-r\right)  $, where $j=1,2,...,p$ and $t_{0}%
\in(-1,1].$

$(b)$ There exists $\alpha\in(0,\delta]$ such that the eigenvalues of the
operator $L_{t}$ for $t\in(t_{0}-\alpha,t_{0}+\alpha)$ lying in $\left(
\mu_{j}(t_{0})-r,\mu_{j}(t_{0})+r\right)  $ are $\lambda_{s_{j-1}%
+1}(t),\lambda_{s_{j-1}+2}(t),...,\lambda_{s_{j}}(t),$ where $s_{0}=0$ and
$s_{j}$ for $j\geq1$ are defined in Notation 2.
\end{theorem}

\begin{proof}
$(a)$ By the definition of $r$, the circle $D_{j}=\left\{  z\in\mathbb{C}%
:\left\vert z-\mu_{j}(t_{0})\right\vert =r\right\}  $ belong to the resolvent
set of the operator $L_{t_{0}}.$ It implies that $\Delta(\lambda,t_{0})\neq0$
for each $\lambda\in D_{j}.$ Since $\Delta(\lambda,t_{0})$ is a continuous
function on $D_{j},$ there exists $a>0$ such that $\left\vert \Delta
(\lambda,t_{0})\right\vert >a$ for all $\lambda\in D_{j}.$ Moreover,
$\Delta(\lambda,t)$ is a polynomial of $e^{i\pi t}$ with entire coefficients.
Therefore, there exists $\delta>0$ such that
\begin{equation}
\left\vert \Delta(\lambda,t)\right\vert >a/2 \tag{42}%
\end{equation}
for all $t\in(t_{0}-\delta,t_{0}+\delta)$ and $\lambda\in D_{j}.$ It means
that $D_{j}$ belong to the resolvent set of $L_{t}$ for all $t\in(t_{0}%
-\delta,t_{0}+\delta).$ It is well-known that
\begin{equation}
\left(  L_{t}-\lambda I\right)  ^{-1}f(x)=\int_{0}^{1}G(x,\xi,\lambda
,t)f(\xi)d\xi, \tag{43}%
\end{equation}
where $G(x,\xi,\lambda,t)$ is the Green's function of $L_{t}$ defined in
Summary 2. On the other hand, it easily follows from Summary 2 and (42) that
there exists $M$ such that
\begin{equation}
\left\vert G(x,\xi,\lambda,t)-G(x,\xi,\lambda,t_{0})\right\vert \leq
M\left\vert t-t_{0}\right\vert \tag{44}%
\end{equation}
for all $x\in\lbrack0,1],$ $\xi\in\lbrack0,1]$, $\lambda\in D_{j}$ and
$t\in(t_{0}-\delta,t_{0}+\delta).$ Therefore using (43), (44) and Summary 2
one can easily verify that $\left(  L_{t}-\lambda I\right)  ^{-1}$ for
$\lambda\in D_{j}$ and the projection
\[
P_{t}=%
{\textstyle\int\nolimits_{D_{j}}}
\left(  L_{t}-\lambda I\right)  ^{-1}d\lambda
\]
continuously depend on $t\in(t_{0}-\delta,t_{0}+\delta).$ This implies that
the operators $L_{t}$ for each $t\in(t_{0}-\delta,t_{0}+\delta)$ have
\ $k_{j}$ eigenvalues in the interval $\left(  \mu_{j}(t_{0})-r,\mu_{j}%
(t_{0})+r\right)  ,$ since $L_{t_{0}}$ have \ $k_{j}$ eigenvalues (counting
the multiplicity) inside $D_{j}$.

$(b)$ Let $b$ be the middle point of the interval $D(N,t_{0}),$ where $D(N,t)$
is defined in (34). It is clear that $b\in D(N,t)$ if $t\in(t_{0}-\delta
,t_{0}+\delta)$ and $\delta$ is a small number. Introduce the notation
$I_{0}:=[b,\mu_{1}(t_{0})-r]$ and $I_{j}:=\left[  \mu_{j}(t_{0})+r,\mu
_{j+1}(t_{0})-r\right]  $ for $j=1,2,...,p.$ Let $\gamma_{j}$ be the closed
curve which enclose the interval $I_{j}$ and belong to the resolvent set of
$L_{t_{0}}.$ Since $L_{t_{0}}$ has no eigenvalues in the intervals $I_{j}$ for
$j=0,1,...,p,$ instead of $D_{j}$ taking the closed curve $\gamma_{j}$ and
arguing as above we obtain that there exists $\delta_{j}$ such that $L_{t}$
for $t\in(t_{0}-\delta_{j},t_{0}+\delta_{j})$ also have no eigenvalues in the
interval $I_{j}$. Therefore, if $\alpha=\min\left\{  \delta,\delta_{1}%
,\delta_{2},\cdot\cdot\cdot,\delta_{p}\right\}  ,$ then the eigenvalues of
$L_{t}$ for $t\in(t_{0}-\alpha,t_{0}+\alpha)$ lying in $\left(  \mu_{j}%
(t_{0})-r,\mu_{j}(t_{0})+r\right)  $ are $\lambda_{s_{j-1}+1}(t),\lambda
_{s_{j-1}+2}(t),...,\lambda_{s_{j}}(t)$ for $j=1,2,...,p.$
\end{proof}

Now using these statements we prove that for each $k\in\mathbb{Z}%
\backslash\left\{  0\right\}  $\ the band function $\lambda_{k}$ is continuous
at each point $t_{0}\in(-1,1].$

\begin{theorem}
For each $k\in\mathbb{Z}\backslash\left\{  0\right\}  $\ the function
$\lambda_{k}$ defined in Notation 1 is continuous at $t_{0}\in(-1,1].$.
\end{theorem}

\begin{proof}
We prove this theorem for $k>0.$ For $k<0$ the proof is similar. Consider any
sequence $\left\{  \left(  \lambda_{k}(t_{p}),t_{p}\right)  :p\in
\mathbb{N}\right\}  $ such that $t_{p}\in(t_{0}-\alpha,t_{0}+\alpha)$ for all
$p\in\mathbb{N}$ and $t_{p}\rightarrow t_{0}$ as $p\rightarrow\infty,$ where
$\alpha$ is defined in Theorem 3$(b)$. If $k\in\lbrack s_{j-1}+1,s_{j}],$ then
by Theorem 3$(b)$ we have $\lambda_{k}(t_{p})\in\left(  \mu_{j}(t_{0}%
)-r,\mu_{j}(t_{0})+r\right)  .$ Let $\left(  \lambda,t_{0}\right)  $ be any
limit point of $\left\{  \left(  \lambda_{k}(t_{p}),t_{p}\right)
:p\in\mathbb{N}\right\}  .$ Since $\Delta$ is continuous with respect to the
pair $(\lambda,t)$ and $\Delta\left(  \lambda_{k}(t_{p}),t_{p}\right)  =0$ for
all $p$ we have $\Delta(\lambda,t_{0})=0.$ It implies that $\lambda$ is an
eigenvalue of $L_{t_{0}}$ lying in $\left(  \mu_{j}(t_{0})-r,\mu_{j}%
(t_{0})+r\right)  $. Hence, by Theorem 3$(b)$ we have $\lambda=\lambda
_{s_{j-1}+1}(t_{0})=\lambda_{s_{j-1}+2}(t_{0})=\cdot\cdot\cdot=\lambda_{s_{j}%
}(t_{0})=\lambda_{k}(t_{0}).$ Thus $\lambda_{k}(t_{p})\rightarrow\lambda
_{k}(t_{0})$ as $p\rightarrow\infty$ for any sequence $\left\{  t_{p}%
:p\in\mathbb{N}\right\}  $ converging to $t_{0}$ and hence $\lambda_{n}$ is
continuous at $t_{0}$.
\end{proof}

By Theorem 4 the bands $\Delta_{s}$ defined in Notation 1 are the intervals of
the spectrum $\sigma(L)$ of $L.$ Since the spectrum $\sigma(L)$ is a closed
set the closure $\overline{\Delta_{s}}$ of $\Delta_{s}$ is the closed
intervals of the spectrum $\sigma(L)$ of $L.$ The space between $\overline
{\Delta_{s}}$ and $\overline{\Delta_{s+1}}$ are the gaps in $\sigma(L).$ Now
we prove the following theorem about the gaps and overlapping problem.

\begin{theorem}
$(a)$ The intervals
\begin{equation}
((2\pi k-\pi)^{n}+\delta_{k}(-1),(2\pi k+\pi)^{n}-\delta_{k}(1)) \tag{45}%
\end{equation}
and
\begin{equation}
((-2\pi k-\pi)^{n}+\delta_{-k}(-1),(-2\pi k+\pi)^{n}-\delta_{-k}(1)) \tag{46}%
\end{equation}
for $k\geq N$ are contained \ respectively in each of the bands%
\begin{equation}
\Delta_{(k-N)m+1},\Delta_{(k-N)m+2},,...,\Delta_{(k-N)m+m} \tag{47}%
\end{equation}
and
\begin{equation}
\Delta_{-(N+k)m},\Delta_{-(N+k)m+1},,...,\Delta_{-(N+k)m+m-1}, \tag{48}%
\end{equation}
where $\delta_{k}(t)$ and $N$ are defined in (5) and Theorem 1.

$(b)$ The spectrum of $L$ contains the intervals $\left[  (2\pi N)^{n}%
,\infty\right)  $ and $(-\infty,(-2\pi N)^{n}].$

$(c)$ The number of gaps in the spectrum of $L$ is not greater than
$m(2N-1)-1.$
\end{theorem}

\begin{proof}
$(a)$ By Theorem 1 and Notation 1 for each $\varepsilon>0$ the eigenvalues
$\lambda_{s}(-1+\varepsilon)$ for $s\in\left\{  \left(  k-N\right)
m+j:j=1,2,...m\right\}  $ of $L_{-1+\varepsilon}$ lie in the interval
$B(k,-1+\varepsilon).$ On the other hand, the eigenvalue $\lambda_{s}(1)$ of
$L_{1}$ lie in the interval $B(k,1).$ Therefore it follows from the definition
of $B(k,t)$ that the interval $((2\pi k-\pi+\varepsilon\pi)^{n}+\delta
_{k}(-1+\varepsilon),(2\pi k+\pi)^{n}-\delta_{k}(1))$ are contained in each of
the bands (47). Now letting $\varepsilon$ tend to zero we obtain that the
interval (45) is contained in each of the bands (47). In the same way we prove
that the interval (46) is contained in each of the bands (48)..

$(b)$ Let $\Gamma(k)$ be the closure of the union of the intervals in (47).
Since the bands in (47) have the common subinterval (45), $\Gamma(k)$ is the
closed interval of $\sigma(L)$ and $\Gamma(k)\subset I(k,t),$ where $I(k,t)$
is defined in (6). Moreover
\begin{equation}
\sigma(L_{1})\cap I(k,1)=\left\{  \lambda_{\left(  k-N\right)  m+j}%
(1):j=1,2,...m\right\}  \subset\Gamma_{k}\subset\sigma(L). \tag{49}%
\end{equation}
Now we prove that the intervals $\Gamma(k)$ and $\Gamma(k+1)$ for $k\geq N$
have a common point. Consider a sequence $\left\{  \left(  \lambda_{s}%
(t_{p}),t_{p}\right)  :p\in\mathbb{N}\right\}  \subset\Gamma(k+1)$ such that
$t_{p}\rightarrow-1$ as $p\rightarrow\infty,$ where $s=\left(  k+1-N\right)
m+j$ for some $j=1,2,...,m.$ Since $\left\{  \lambda_{s}(t_{p}):p\in
\mathbb{N}\right\}  $ is a bounded sequence there exists a convergent
subsequence $\left\{  \lambda_{s}(t_{p_{j}}):j\in\mathbb{N}\right\}  $ of this
sequence converging to $a\in\Gamma(k+1)$. Let us we prove that $a\in\Gamma
(k)$. Since $\Delta$ is a continuous function with respect to the pair
$(\lambda,t)$ and $\Delta\left(  \lambda_{s}(t_{p_{j}}),t_{p_{j}}\right)  =0$
for all $j$ we have $\Delta(a,-1)=0.$ It implies that $a$ is an eigenvalue of
$L_{1}$ lying in $I(k,1).$ Therefore it follows from (49) that $a\in\Gamma
(k)$. Thus $\Gamma(k+1)$ and $\Gamma_{k}$ is connected for $k\geq N.$ In the
same way we prove that $\Gamma(N)$ contains some eigenvalue $\lambda_{s}(1)$
of $L_{1}$ lying in $A(N,1).$ If $\lambda_{s}(1)\in\Gamma(N)$ for some
$s\in\left\{  -1,-2,...,-\left(  2N-1\right)  m\right\}  $ then $\lambda
_{-1}(1)\in\Gamma_{N}.$ Hence, in any case the last inclusion holds. Thus we
have
\begin{equation}
\lambda_{-1}(1)\in\Gamma(N),\text{ }\Gamma(k)\cap\Gamma(k+1)\neq
\varnothing\tag{50}%
\end{equation}
for $k\geq N.$ Since $\lambda_{-1}(1)\in A(N,1)=[(-2\pi N+2\pi)^{n},(2\pi
N)^{n})$ we see that $[(2\pi N)^{n},\infty)\in\sigma(L).$ In the same way we
prove that
\begin{equation}
\lambda_{-\left(  2N-1\right)  m}(1)\in\Gamma(-N),\text{ }\Gamma(-k)\cap
\Gamma(-k-1)\neq\varnothing\tag{51}%
\end{equation}
and $(-\infty,-(2\pi N)^{n}]\in\sigma(L)$.

$(c)$ It follows from $(b)$ that there may be gaps only between $\overline
{\Delta_{s}}$ and $\overline{\Delta_{s-1}}$ $\ $\ for $s\in\left\{
-1,-2,...,-\left(  2N-1\right)  m+1\right\}  .$ Therefore the number of the
gaps in the spectrum of $L$ is not greater than $\left(  2N-1\right)  m-1$.
\end{proof}

Now we consider the spectrum of $L$ when $M$ is a small number.

\begin{theorem}
If (35) holds, then the intervals $\left(  \pi^{n}+\frac{3}{10}\pi^{n},\left(
3\pi\right)  ^{n}-\frac{3}{10}3^{n-2}\pi^{n}\right)  ,$ $(\left(  -\pi\right)
^{n}+\frac{1}{5}\pi^{n},\pi^{n}-\frac{1}{5}\pi^{n})$ and $\left(  -\left(
3\pi\right)  ^{n}+\frac{3}{10}3^{n-2}\pi^{n},-\pi^{n}-\frac{3}{10}\pi
^{n}\right)  $ are contained in each of the bands $\Delta_{-s}$ for
$s\in\left\{  1,2,...m\right\}  ,$ $s\in\left\{  m+1,m+2,...2m\right\}  $ and
$s\in\left\{  2m+1,2m+2,...3m\right\}  $ respectively. The spectrum of $L$ is
$(-\infty,\infty).$
\end{theorem}

\begin{proof}
Using Theorem 2 and Notation1 and repeating the proof of Theorem 5$(a)$ we get
the proof of the first statement. If (35) holds then $N=2$. Therefore (50) and
(51) holds for $N=2.$ Now introduce the notations%
\begin{equation}
\Gamma(1)=%
{\textstyle\bigcup\limits_{s=1}^{m}}
\overline{\Delta_{-s}},\text{ }\Gamma(0)=%
{\textstyle\bigcup\limits_{s=m+1}^{2m}}
\overline{\Delta_{-s}},\text{ }\Gamma(-1)=%
{\textstyle\bigcup\limits_{s=2m+1}^{3m}}
\overline{\Delta_{-s}}. \tag{52}%
\end{equation}
Arguing as in the proof of Theorem 5$(b)$ we obtain that $\Gamma(-1),$
$\Gamma(0)$ and $\Gamma(1)$ are the closed intervals. Moreover it follows from
(50) and (51) for $N=2$ that $\lambda_{-1}\in\Gamma(2)$ and $\lambda_{-\left(
2N-1\right)  m}(1)\in\Gamma(-2).$ This with (52) implies that $\Gamma(1)\cap$
$\Gamma(2)\neq\varnothing$ and $\Gamma(-1)\cap$ $\Gamma(-2)\neq\varnothing.$
Moreover, repeating the proof of (50) and (51), we get $\Gamma(1)\cap$
$\Gamma(0)\neq\varnothing$ and $\Gamma(-1)\cap$ $\Gamma(0)\neq\varnothing.$
Thus the intervals $\Gamma(k)$ and $\Gamma(k+1)$ have a common point for all
$k\in\mathbb{Z}$. It means that there are no gaps in the spectrum and
$\sigma(L)=(-\infty,\infty).$
\end{proof}

\end{document}